\documentclass[twoside]{amsart}
\usepackage{amsmath,amssymb,amsbsy,amstext,amsxtra,latexsym}
\usepackage{graphics, epsfig}

\newtheorem{theorem}{Theorem}
\newtheorem{lemma}{Lemma}[section]
\newtheorem{proposition}{Proposition}[section]

\theoremstyle{remark}

\def\mP{{\mathbb P}}
\def\mQ{{\mathbb Q}}

\def\mZ{{\mathbb Z}}

\def\n{\noindent}
\def\b{\bigskip}
\def\m{\medskip}
\def\t{\tilde}

\begin{document}

\title[A complex surface of general type with $p_g=0, K^2=2$ and
       $H_1 = {\mZ}/2\mZ,\, {\mZ}/3\mZ$]{A complex surface of general type
       with $p_g=0,$ $K^2=2$ and $H_1 = {\mZ}/2\mZ,\, {\mZ}/3\mZ$}

\author{Yongnam Lee and Jongil Park}

\address{Department of Mathematics, Sogang University,
         Sinsu-dong, Mapo-gu, Seoul 121-742, Korea}

\email{ynlee@sogang.ac.kr}

\address{Department of Mathematical Sciences, Seoul National University,
         San 56-1, Sillim-dong, Gwanak-gu, Seoul 151-747, Korea}

\email{jipark@math.snu.ac.kr}

\date{August 31, 2008}



\begin{abstract}
 As the sequel to~\cite{LP}, we construct a minimal complex surface of general
 type with $p_g=0$, $K^2=2$ and $H_1=\mZ/2\mZ$ using a rational blow-down
 surgery and $\mQ$-Gorenstein smoothing theory. We also present an example of
 $p_g=0, K^2=2$ and $H_1=\mZ/3\mZ$.
\end{abstract}

\maketitle

 One of the fundamental problems in the classification of complex
 surfaces is to find a new family of surfaces of general type with $p_g=0$
 and $K^2 >0$. In this paper we construct new complex surfaces
 of general type with $p_g=0$ and $K^2=2$.
 The first example of a minimal complex surface of general type with
 $p_g=0$ and $K^2 =2$ was constructed by Campedelli~\cite{Cam} in the
 $1930$'s as a ramified double cover of $\mP^2$; more precisely as the double
 cover of $\mP^2$ branched along a reducible curve of degree $10$ with
 $6 \, [3, 3]$ points not lying on a conic.
 Nowadays minimal surfaces of general type with
 $p_g=0$ and $K^2=2$ are called (numerical) Campedelli surfaces.
 For Campedelli surface $X$,
 the number of elements in the torsion subgroup of $H^2(X; \mZ)$ is
 bounded by $9$~\cite{Reid}. Although many families of non-simply connected
 complex surfaces of general type with $p_g =0$ and $K^2=2$ have been
 constructed (refer to Chapter VII, \cite{BHPV}) and the
 classifications are completed for some torsion groups~\cite{MP},
 until now it is not known whether there is a minimal complex surface
 of general type with $p_g=0$, $K^2 =2$ and $H_1=\mZ/2\mZ$ (or $\mZ/3\mZ)$.

 Recently we constructed a simply connected surface of general
 type with $p_g=0$ and $K^2 =2$ using a rational blow-down surgery and
 $\mQ$-Gorenstein smoothing theory~\cite{LP}.
 In this paper we continue to construct minimal complex surfaces of general
 type with $p_g=0$, $K^2=2$ and $H_1=\mZ/2\mZ,\, \mZ/3\mZ$
 using the same technique.
 The first key ingredient of this paper is to find right rational
 surfaces $Z$ which make it possible to get such complex surfaces.
 Once we have a right rational surface $Z$, then
 we can obtain a minimal complex surface of general type with $p_g=0$
 and $K^2 =2$ by applying a rational blow-down surgery and
 $\mQ$-Gorenstein smoothing theory developed in~\cite{LP} to $Z$.
 And then we show that the surface has $H_1 = \mZ/2\mZ$ (or $\mZ/3\mZ)$,
 which is the second key ingredient of this article.
 Since almost all the proofs except the computation of homology groups
 are basically the same as the proofs of the main construction in~\cite{LP},
 we only explain how to construct such surfaces and
 how to compute the first homology groups.
 The main results of this paper are the following

\begin{theorem}
\label{thm-main}
 There exists a minimal complex surface of general type with
 $p_g=0$, $K^2 =2$ and $H_1=\mZ/2\mZ$.
\end{theorem}

\begin{theorem}
\label{thm-main2}
 There exists a minimal complex surface of general type with
 $p_g=0$, $K^2 =2$ and $H_1=\mZ/3\mZ$.
\end{theorem}

\m

 {\em Acknowledgements}.
 The authors would like to thank Heesang Park for pointing out some
 errors in the computation of Lemma~\ref{lem-4}.
 Yongnam Lee was supported by the grant No. R01-2007-000-10948-0
 from the KOSEF.
 Jongil Park was supported by SBS Foundation Grant in 2007 and
 he also holds a joint appointment in the Research Institute of Mathematics,
 Seoul National University.

\b

\section{The main construction of a surface with $K^2 =2$ and $H_1=\mZ/2\mZ$}
\label{sec-1}

 We begin with a rational elliptic surface
 $Y=\mP^2\sharp 9\overline{\mP}^2$ which has one $\tilde E_6$-singular
 fiber, one $I_2$-singular fiber, and two nodal singular fibers
 used in Section 3 in~\cite{LP} (Figure~\ref{Fig-Y}).

 \begin{figure}[hbtb]
 \begin{center}
 \setlength{\unitlength}{1mm}
 \includegraphics[height=3.5 cm]{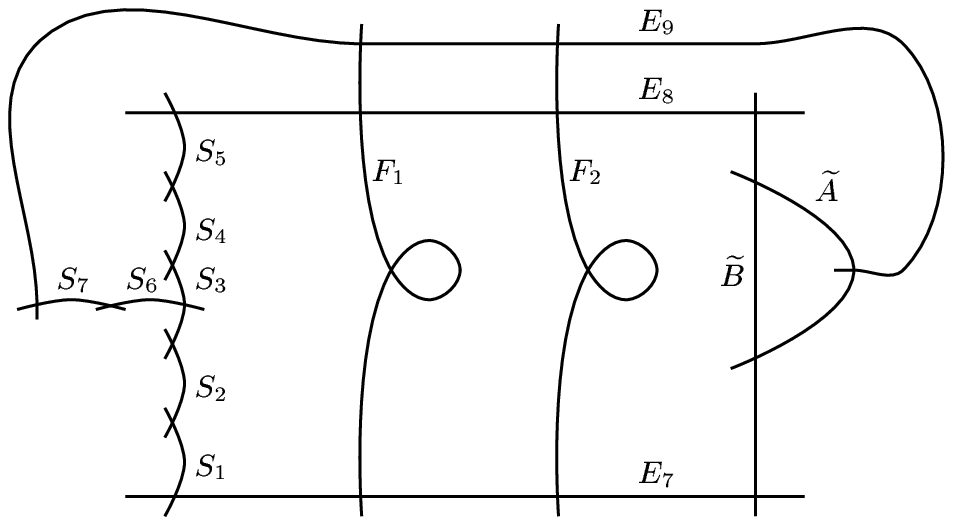}
 \end{center}
 \vspace{-1 em}
 \caption{A rational elliptic surface $Y=\mP^2\sharp 9\overline{\mP}^2$}
 \label{Fig-Y}
 \end{figure}

 {\em Notations.} We denote a line in $\mP^2$ by $H$
 and exceptional curves in $Y=\mP^2\sharp 9\overline{\mP}^2$
 by $E_1, E_2, \ldots, E_9$. Equivalently we use the same notation
 $H, E_1, E_2, \ldots, E_9$ for the standard generators of $H_2(Y; \mZ)$
 which represent a line and exceptional curves in $Y$ respectively.
 We also denote the rational curves lying in the $\t E_6$-singular fiber
 by $S_1, S_2, \ldots, S_7$ and two nodal fibers by $F_1, F_2$ and two rational
 curves lying in the $I_2$-singular fiber by $\t A, \t B$ respectively.
 In fact $\t A$ and $\t B$ are proper transforms of a line $A$ and a
 conic $B$ lying in $\mP^2$ respectively.

\m

\n{\bf Main Construction.} We first blow up at two singular points
 in the nodal fibers $F_1, F_2$ on $Y$.
 Then the proper transforms $\t F_1, \t F_2$ of $F_1, F_2$ will be
 rational $(-4)$-curves whose homology classes are $[\t F_1]=[F_1]-2e_1$
 and $[\t F_2]=[F_2]-2e_2$, where $e_1, e_2$ are new exceptional curves
 in $Y\sharp 2\overline{\mP}^2$ coming from two blowing ups.
 Next, we blow up six times at the intersection points
 between two sections $E_7, E_8$ and $\t F_1, \t F_2, \t B$.
 It makes the self-intersection number of the proper transforms
 $E_7$, $E_8$ and $\t B$ to be $-4$ respectively.
 Let us denote six new exceptional curves arising from six times
 blowing ups by $e_3, e_4, \ldots, e_8$ respectively.
 Now we blow up twice successively at the intersection point between
 the proper transform of $E_7$ and the exceptional curve $e_3$
 in the total transform of $F_1$.
 It makes a chain of $\mP^1$'s,
 ${\overset{-6}{\circ}}-{\overset{-2}{\circ}}-{\overset{-2}{\circ}}$,
 lying in the total transform of $F_1$.
 Let us denote two new exceptional curves arising from twice blowing
 ups by $e_{9}, e_{10}$ respectively.
 We blow up again four times successively at the intersection point
 between the proper transform of $E_7$ and the exceptional curve
 $e_4$ in the total transform of $F_2$, so that a chain of $\mP^1$'s,
 ${\overset{-6}{\circ}}-{\overset{-2}{\circ}}-{\overset{-2}{\circ}}
 -{\overset{-2}{\circ}}-{\overset{-2}{\circ}}$,
 lies in the total transform of $F_2$.
 Let us denote four new exceptional curves arising from four times blowing
 ups at this step by $e_{11}, e_{12}, e_{13}, e_{14}$ respectively.
 We note that it makes the self-intersection number of the proper
 transform of $E_7$ to be $-10$.
 Then we blow up twice successively
 at the intersection point between rational $(-2)$-curve in the end
 of linear chain
 ${\overset{-6}{\circ}}-{\overset{-2}{\circ}}-{\overset{-2}{\circ}}
 -{\overset{-2}{\circ}}-{\overset{-2}{\circ}}$
 and the exceptional curve $e_{14}$.
 Let us denote two new exceptional curves arising from twice blowing
 ups at this step by $e_{15}, e_{16}$ respectively.
 It changes
 ${\overset{-6}{\circ}}-{\overset{-2}{\circ}}-{\overset{-2}{\circ}}
 -{\overset{-2}{\circ}}-{\overset{-2}{\circ}}$
 to
 ${\overset{-6}{\circ}}-{\overset{-2}{\circ}}-{\overset{-2}{\circ}}
 -{\overset{-2}{\circ}}-{\overset{-4}{\circ}}$,
 and it produces a chain of $\mP^1$'s,
 ${\overset{-2}{\circ}}-{\overset{-2}{\circ}}-{\overset{-10}{\circ}}
  -{\overset{-2}{\circ}}-{\overset{-2}{\circ}} -{\overset{-2}{\circ}}
  -{\overset{-2}{\circ}}-{\overset{-2}{\circ}}-{\overset{-4}{\circ}}$,
 which contains the proper transform of two sections $E_7, E_8$ and
 a linear chain of $\mP^1$'s in the $\t E_6$-singular fiber.
 We also blow up twice successively at the intersection point between
 rational $(-6)$-curve in
 ${\overset{-6}{\circ}}-{\overset{-2}{\circ}}-{\overset{-2}{\circ}}
 -{\overset{-2}{\circ}}-{\overset{-4}{\circ}}$
 and the exceptional curve $e_2$ appeared by the blowing up at the
 singular point of one nodal fiber $F_2$.
 Let us denote two new exceptional curves arising from twice blowing
 ups at this step by $e_{17}, e_{18}$ respectively.
 Then the chain
 ${\overset{-6}{\circ}}-{\overset{-2}{\circ}}-{\overset{-2}{\circ}}
 -{\overset{-2}{\circ}}-{\overset{-4}{\circ}}$
 changes to
 ${\overset{-2}{\circ}}-{\overset{-2}{\circ}}-{\overset{-8}{\circ}}
 -{\overset{-2}{\circ}}-{\overset{-2}{\circ}}
  -{\overset{-2}{\circ}}-{\overset{-4}{\circ}}$ by adding two new
 rational $(-2)$-curves.
 Finally, we have a rational surface $Z:= Y \sharp 18{\overline \mP}^2$
 which contains four disjoint configurations:
 $C_{22,15}={\overset{-2}{\circ}}-{\overset{-2}{\circ}}-{\overset{-10}{\circ}}
 -{\overset{-2}{\circ}}-{\overset{-2}{\circ}}-{\overset{-2}{\circ}}
 -{\overset{-2}{\circ}}-{\overset{-2}{\circ}}-{\overset{-4}{\circ}}$,
 $C_{4,1}={\overset{-6}{\circ}}-{\overset{-2}{\circ}}-{\overset{-2}{\circ}}$,
 $C_{16,11}={\overset{-2}{\circ}}-{\overset{-2}{\circ}}-{\overset{-8}{\circ}}
 -{\overset{-2}{\circ}}-{\overset{-2}{\circ}}-{\overset{-2}{\circ}}
 -{\overset{-4}{\circ}}$, and $C_{2,1}={\overset{-4}{\circ}}$
 (Figure~\ref{Fig-Z}).

\begin{figure}[hbtb]
 \begin{center}
 \setlength{\unitlength}{1mm}
 \includegraphics[height=4 cm]{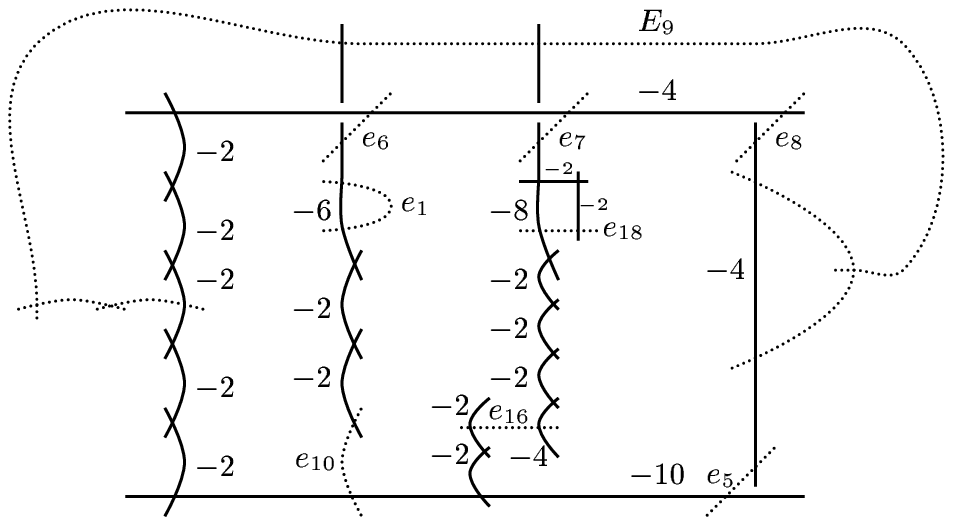}
 \end{center}
 \vspace{-1 em}
 \caption{A rational surface $Z= Y \sharp 18{\overline \mP}^2$}
 \label{Fig-Z}
\end{figure}

\m

 {\em Notations.} We use the same notation $C_{p,q}$ for both a
 smooth $4$-manifold obtained by plumbing
 disk bundles over the $2$-sphere instructed by
 $\underset{u_{k}}{\overset{-b_k}{\circ}}-\underset{u_{k-1}}
 {\overset{-b_{k-1}}{\circ}}-\cdots -\underset{u_2}{\overset{-b_{2}}{\circ}}
 -\underset{u_1}{\overset{-b_{1}}{\circ}}$
 and a linear chain of $2$-spheres, $\{u_{k}, u_{k-1},\ldots, u_1 \}$.
 Here $\frac{p^{2}}{pq-1} =[b_{k},b_{k-1}, \ldots, b_{1}]$
 is a continued fraction with all $b_{i} \geq 2$ uniquely determined
 by $p, q$, and $u_{i}$ represents an embedded $2$-sphere as well as
 a disk bundle over $2$-sphere whose Euler number is $-b_{i}$.

 \m

 Then we contract these four disjoint chains of $\mP^1$'s from $Z$
 so that it produces a normal projective
 surface, denoted by $X$, with four permissible singularities.
 Using the same technique as in~\cite{LP}, we are able to prove that
 $X$ has a $\mQ$-Gorenstein smoothing and a general fiber $X_t$ of
 the $\mQ$-Gorenstein smoothing is a minimal
 complex surface of general type with $p_g=0$ and $K^2=2$.
 Furthermore, the general fiber $X_t$ is diffeomorphic to a rational
 blow-down $4$-manifold $Z_{22,16,4,2}$ which is obtained from
 $Z= Y \sharp 18{\overline \mP}^2$ by replacing four disjoint
 configurations $C_{22,15},\, C_{16,11},\, C_{4,1}$ and $C_{2,1}$
 with corresponding rational balls
 $B_{22,15},\, B_{16,11},\, B_{4,1}$ and $B_{2,1}$ respectively.
 In the next section we will prove that the rational
 blow-down $4$-manifold $Z_{22,16,4,2}$ has
 $H_1(Z_{22,16,4,2}; \mZ)=\mZ/2\mZ$.
 Hence we obtain a main result - Theorem~\ref{thm-main}.

\b

\section{Proof of $H_1(Z_{22,16,4,2}; \mZ)=\mZ/2\mZ$}
\label{sec-2}

 In this section we compute the first homology group of a rational blow-down
 $4$-manifold $Z_{22,16,4,2}$ using geometric arguments and some elementary
 homology sequences.

 First note that the rational surface $Z= Y \sharp 18{\overline \mP}^2$
 contains four disjoint configurations
 - $C_{22,15},\, C_{16,11},\, C_{4,1}$ and $C_{2,1}$.
 Let us decompose the rational surface $Z$ into
 \[Z=Z_0 \cup\{C_{22,15}\cup C_{16,11} \cup C_{4,1} \cup C_{2,1}\} .\]
 Then the rational blow-down $4$-manifold $Z_{22,16,4,2}$
 can be decomposed into
 \[Z_{22,16,4,2}=
   Z_0 \cup\{B_{22,15}\cup B_{16,11} \cup B_{4,1} \cup B_{2,1}\} . \]
 Furthermore, one can also obtain the rational blow-down $4$-manifold
 $Z_{22,16,4,2}$ by two steps, i.e.
 one can first construct a new simply connected smooth $4$-manifold
 $Z'_{22,16,4}$ by performing a rational blow-down surgery along three
 configurations $C_{22,15},\, C_{16,11}$ and $C_{4,1}$ from $Z$, and then
 perform a rational blow-down surgery along $C_{2,1}$ from
 $Z'_{22,16,4}$ to get $Z_{22,16,4,2}$.
 Let us decompose $Z'_{22,16,4}= W \cup C_{2,1}$
 with $W=Z_0 \cup\{B_{22,15}\cup B_{16,11} \cup B_{4,1}\}$.
 Then $Z_{22,16,4,2}$ can also be decomposed into
 $Z_{22,16,4,2}= W \cup B_{2,1}$.

\begin{proposition}
\label{pro-1}
 $H_1(Z_{22,16,4,2}; \mZ) = 0$ or $H_1(Z_{22,16,4,2}; \mZ) = \mZ/2\mZ$.
\end{proposition}

\begin{proof}
 We first consider the following exact homology sequence for a pair
 $(W, \partial{W})$:
 \[\longrightarrow H_2(W, \partial{W};\mZ)
   \stackrel{\partial_{\ast}}{\longrightarrow} H_1(\partial{W}; \mZ)
   \stackrel{i_{\ast}}{\longrightarrow} H_1(W; \mZ) \longrightarrow 0 \]
 Note that $\partial{W}=L(4,-1)$ and a generator of
 $H_1(\partial{W}; \mZ)=\mZ/4\mZ$ can be represented by a normal
 circle, say $\t \delta$, of a disk bundle $C_{2,1}$ over $(-4)$-curve $\t B$.
 Since the image of a homology class $[\t A|_{W}] \in
 H_2(W, \partial{W};\mZ)$ under the boundary homomorphism
 $\partial_{\ast}$ is $2\t \delta \in H_1(\partial{W}; \mZ)=\mZ/4\mZ$,
 we have either $H_1(W; \mZ) = 0$ or $H_1(W; \mZ) = \mZ/2\mZ$.

 Next we consider the Mayer-Vietoris sequence for a triple
 $(Z_{22,16,4,2}; W, B_{2,1})$:
\small{
\begin{equation*}
  H_2(Z_{22,16,4,2};\mZ)
   \stackrel{\partial_{\ast}}{\rightarrow} H_1(L(4,-1); \mZ)
   \stackrel{i_{\ast}\oplus j_{\ast}}{\longrightarrow}
   H_1(W; \mZ)\oplus H_1(B_{2,1}; \mZ) \rightarrow
   H_1(Z_{22,16,4,2}; \mZ) \rightarrow 0
\end{equation*}}
\normalsize
 Since $j_{\ast}: H_1(L(4,-1); \!\mZ) \rightarrow H_1(B_{2,1}; \!\mZ)$
 is surjective, we have either $H_1(Z_{22,16,4,2}; \!\mZ)$ $=0$ or
 $H_1(Z_{22,16,4,2}; \mZ)= \mZ/2\mZ$ depending on the space $H_1(W; \mZ)$.
 In other words, if $H_1(W; \mZ)=0$, then we would have
 $H_1(Z_{22,16,4,2}; \mZ)=0$  and, if $H_1(W; \mZ)=\mZ/2\mZ$,
 then we would have $H_1(Z_{22,16,4,2}; \mZ)=\mZ/2\mZ$.
\end{proof}

 Before we prove that $H_1(Z_{22,16,4,2}; \mZ)=\mZ/2\mZ$, we introduce several
 lemmas which are critical in the computation of $H_1(Z_{22,16,4,2}; \mZ)$.
 Let us first consider the following two sets of homology classes lying in
 $H_2(Z; \mZ)$:
 ${\mathcal D}=\{\t A,\, E_3-E_6,\, E_6-E_9 \}$ and
 ${\mathcal E}=\{E_9,\, e_1,\, e_5,\, e_6,\, e_7,\, e_8,\, e_{10},\,
 e_{16},\, e_{18} \}$.
 Then one can easily get the following lemmas.

\begin{lemma}
\label{lem-1}
 (i) The homology classes of $\{u_i\,|\, 1\leq i \leq 9\}$ in $C_{22,15}$
 can be represented by
 $\{e_{15}-e_{16},\, e_{14}-e_{15},\, E_7-e_3-e_4-e_5-e_9-\cdots -e_{14},\,
 E_4-E_7,\, E_1-E_4,\, H-E_1-E_2-E_3,\, E_2-E_5,\, E_5-E_8,\,
 E_8-e_6-e_7-e_8\}$. \\
 (ii) The homology classes of $\{u_i\,|\, 1\leq i \leq 7\}$ in $C_{16,11}$
 can be represented by
 $\{e_{17}-e_{18},\, e_2-e_{17},\, F_2-2e_2-e_4-e_7-e_{17}-e_{18},\,
 e_4-e_{11},\, e_{11}-e_{12},\, e_{12}-e_{13},\,
 e_{13}-e_{14}-e_{15}-e_{16}\}$. \\
 (iii) The homology classes of $\{u_i\,|\, 1\leq i \leq 3\}$ in $C_{4,1}$
 can be represented by
 $\{F_1-2e_1-e_3-e_6,\, e_3-e_9,\, e_9-e_{10} \}$. \\
 (iv) The homology classes of $u_1$ in $C_{2,1}$
 can be represented by $\{\t B-e_5-e_8\}$.
\end{lemma}

\begin{lemma}
\label{lem-2}
 The set of homology classes representing all generators in
 $C_{22,15} \cup C_{16,11} \cup C_{4,1} \cup C_{2,1} \cup {\mathcal D}
 \cup {\mathcal E}$ spans $H_2(Z; \mZ)$.
\end{lemma}

\begin{lemma}
\label{lem-3}
 $H_2(Z_0, \partial{B_{22,15}} \cup \partial{B_{16,11}} \cup \partial{B_{4,1}}
 \cup \partial{B_{2,1}}; \mZ)$ is spanned by the images of homology classes
 representing all generators in
 $C_{22,15} \cup C_{16,11} \cup C_{4,1} \cup C_{2,1} \cup {\mathcal D}
 \cup {\mathcal E}$ under a composition of homomorphisms
 $H_2(Z; \mZ) \rightarrow
 H_2(Z, C_{22,15} \cup C_{16,11} \cup C_{4,1} \cup C_{2,1}; \mZ) \cong
 H_2(Z_0, \partial{B_{22,15}} \cup \partial{B_{16,11}} \cup \partial{B_{4,1}}
 \cup \partial{B_{2,1}}; \mZ)$.
\end{lemma}

\begin{proof}
 Since an induced homomorphism $H_2(Z; \mZ) \rightarrow
 H_2(Z, C_{22,15} \cup C_{16,11} \cup C_{4,1} \cup C_{2,1}; \mZ)$
 by an inclusion is surjective and
 $H_2(Z, C_{22,15} \cup C_{16,11} \cup C_{4,1} \cup C_{2,1}; \mZ)$
 is isomorphic to $H_2(Z_0, \partial{C_{22,15}} \cup \partial{C_{16,11}}
 \cup \partial{C_{4,1}} \cup \partial{C_{2,1}}; \mZ)
 = H_2(Z_0, \partial{B_{22,15}} \cup \partial{B_{16,11}} \cup
 \partial{B_{4,1}} \cup \partial{B_{2,1}}; \mZ)$ by an excision principle,
 the statement follows from Lemma~\ref{lem-2}.
\end{proof}

\begin{lemma}
\label{lem-4}
 Suppose that $\partial_{\ast}: H_2(Z_0, \partial{B_{22,15}} \cup
 \partial{B_{16,11}} \cup \partial{B_{4,1}} \cup \partial{B_{2,1}};\mZ)
 \rightarrow H_1(\partial{B_{22,15}} \cup \partial{B_{16,11}}
 \cup \partial{B_{4,1}} \cup \partial{B_{2,1}};\mZ) =
 H_1(\partial{B_{22,15}};\mZ) \oplus H_1(\partial{B_{16,11}};\mZ)
 \oplus H_1(\partial{B_{4,1}};\mZ) \oplus H_1(\partial{B_{2,1}};\mZ)$
 $= \mZ/{22^2}\mZ \oplus \mZ/{16^2}\mZ \oplus \mZ/{4^2}\mZ \oplus \mZ/{2^2}\mZ$
 is a homomorphism induced by a boundary map $\partial: (Z_0, \partial Z_0)
 \rightarrow \partial Z_0$.
 And let $i_{\ast}: H_1(\partial{B_{22,15}} \cup
 \partial{B_{16,11}} \cup \partial{B_{4,1}} \cup \partial{B_{2,1}};\mZ)
 = \mZ/{22^2}\mZ \oplus \mZ/{16^2}\mZ \oplus \mZ/{4^2}\mZ \oplus \mZ/{2^2}\mZ
 \rightarrow H_1(B_{22,15} \cup B_{16,11} \cup B_{4,1} \cup B_{2,1};\mZ)
 = \mZ/22\mZ \oplus \mZ/16\mZ \oplus \mZ/4\mZ \oplus \mZ/2\mZ$
 be a homomorphism induced by an inclusion $i$.
 Then we have \\
 $(0) \ \partial_{\ast}(u_i) = (0,0,0,0)
         \stackrel{i_{\ast}}{\longrightarrow} (0,0,0,0)$ for any class
         $u_i \in C_{22,15} \cup C_{16,11} \cup C_{4,1} \cup C_{2,1}$ \\
 $(i) \ \partial_{\ast}(\t A) = (0,0,0,2)
         \stackrel{i_{\ast}}{\longrightarrow} (0,0,0,0)$ \\
 $(ii)\ \partial_{\ast}(E_3-E_6) = (10,0,0,0)
         \stackrel{i_{\ast}}{\longrightarrow} (10,0,0,0)$ \\
 $(iii)\ \partial_{\ast}(E_6-E_9) = (0,0,0,0)
         \stackrel{i_{\ast}}{\longrightarrow} (0,0,0,0)$ \\
 $(iv)\ \partial_{\ast}(E_9) = (0,13,3,0)
         \stackrel{i_{\ast}}{\longrightarrow} (0,13,3,0)$ \\
 $(v) \  \partial_{\ast}(e_1) = (0,0,6,0)
         \stackrel{i_{\ast}}{\longrightarrow} (0,0,2,0)$ \\
 $(vi)\  \partial_{\ast}(e_5) = (19,0,0,1)
         \stackrel{i_{\ast}}{\longrightarrow} (19,0,0,1)$ \\
 $(vii)\ \partial_{\ast}(e_6) = (1,0,3,0)
         \stackrel{i_{\ast}}{\longrightarrow} (1,0,3,0)$ \\
 $(viii)\ \partial_{\ast}(e_7) = (1,13,0,0)
         \stackrel{i_{\ast}}{\longrightarrow} (1,13,0,0)$ \\
 $(ix)\ \partial_{\ast}(e_8) = (1,0,0,1)
         \stackrel{i_{\ast}}{\longrightarrow} (1,0,0,1)$ \\
 $(x) \ \partial_{\ast}(e_{10}) = (19,0,1,0)
         \stackrel{i_{\ast}}{\longrightarrow} (19,0,1,0)$ \\
 $(xi)\  \partial_{\ast}(e_{16}) = (329,1,0,0)
         \stackrel{i_{\ast}}{\longrightarrow} (21,1,0,0)$ \\
 $(xii)\ \partial_{\ast}(e_{18}) = (0,188,0,0)
         \stackrel{i_{\ast}}{\longrightarrow} (0,12,0,0)$.
\end{lemma}

\begin{proof}
 $(0)$ and $(iii)$ follow from the fact that they do not intersect
 with $\partial Z_0$.
 For the rest, we choose generators
 $\{\alpha=(1,0,0,0), \beta=(0,1,0,0), \gamma=(0,0,0,1), \delta=(0,0,0,1)\}$
 of $H_1(B_{22,15};\mZ) \oplus H_1(B_{16,11};\mZ) \oplus H_1(B_{4,1};\mZ)
 \oplus H_1(B_{2,1};\mZ)$ so that $\alpha, \beta, \gamma$ and $\delta$ are
 represented by circles $\partial C_{22,15} \cap e_6$ (equivalently $\partial
 C_{22,15} \cap e_7$ or $\partial C_{22,15} \cap e_8$),
 $\partial C_{16,11} \cap e_{16}$, $\partial C_{4,1} \cap e_{10}$
 and $\partial C_{2,1} \cap e_{5}$ (equivalently $\partial C_{2,1} \cap
 e_8$), respectively.
 Then one can easily see that the rest of computation follows from
 Figure~\ref{Fig-Z}. For example, we can compute $(iv)$ as follows:
 Note that an exceptional curve $E_9$ intersects with $(-8)$-curve
 in $C_{16,11}$ and it also intersects with $(-6)$-curve in $C_{4,1}$.
 Since $\partial_{*}(E_9) (=$ a normal circle of $(-8)$-curve)
 is $13 \beta \in H_1(\partial B_{16,11})$
 and $\partial_{*}(E_9) (=$ a normal circle of $(-6)$-curve)
 is $3 \gamma \in H_1(\partial B_{4,1})$, we have
 $\partial_{*}(E_9)=(0,13,3,0)$.
 The images of $i_{\ast}$ follow from the fact that
 $i_{\ast}: H_1(\partial B_{p,q})= \mZ/{p^2}\mZ
 \rightarrow H_1(B_{p,q})= \mZ/p\mZ$ sends a generator to a generator,
 i.e. $i_{\ast}(1)=1$.

\end{proof}

 Finally, using Lemma~\ref{lem-4} and Proposition~\ref{pro-1} above,
 we compute $H_1(Z_{22,16,4,2}; \mZ)$.

\begin{proposition}
\label{pro-2}
 $H_1(Z_{22,16,4,2};\mZ) = \mZ/2\mZ.$
\end{proposition}

\begin{proof}
 First let us consider the following commutative diagram between two
 exact homology sequences with $\mZ$-coefficients for pairs
 $(Z_0, \partial{Z_0})$
 and $(Z_{22,16,4,2}, B_{22,15} \cup B_{16,11} \cup B_{4,1} \cup B_{2,1})$:
\tiny{
\begin{eqnarray*}
 & \hspace{-1 em} H_2(Z_0, \partial{B_{22,15}} \cup
 \partial{B_{16,11}} \cup \partial{B_{4,1}} \cup \partial{B_{2,1}})
 \stackrel{\partial_{\ast}}{\rightarrow} H_1(\partial{B_{22,15}} \cup
 \partial{B_{16,11}} \cup \partial{B_{4,1}} \cup \partial{B_{2,1}})
 \stackrel{j_{\ast}}{\rightarrow} H_1(Z_0) \rightarrow 0 \\
 & \hspace{-10 em} i_{\ast} \downarrow \cong \hspace{15 em} \downarrow i_{\ast} \\
 & \hspace{-1 em} H_2(Z_{22,16,4,2}, B_{22,15} \cup B_{16,11} \cup B_{4,1}
 \cup B_{2,1}) \stackrel{\partial_{\ast}}{\rightarrow}
 H_1(B_{22,15} \cup B_{16,11} \cup B_{4,1} \cup B_{2,1})
 \stackrel{j_{\ast}}{\rightarrow} H_1(Z_{22,16,4,2}) \rightarrow 0
\end{eqnarray*}}

\normalsize
 \n where $i_{\ast}$ and $j_{\ast}$ are induced homomorphisms by
 inclusions $i$ and $j$ respectively,
 and $\partial_{\ast}$ is an induced homomorphism by a boundary map $\partial$.
 Note that the first $i_{\ast}$ is an isomorphism by an excision
 principle and the second $i_{\ast}$ is surjective.

 Next, we try to show that
 $\partial_{\ast}: H_2(Z_{22,16,4,2}, B_{22,15} \cup B_{16,11} \cup B_{4,1}
 \cup B_{2,1};\mZ) \longrightarrow H_1(B_{22,15} \cup B_{16,11} \cup B_{4,1}
 \cup B_{2,1};\mZ)$ is not surjective. Equivalently, we prove that
 the composition $H_2(Z_0, \partial B_{22,15}
 \cup \partial B_{16,11} \cup \partial B_{4,1} \cup \partial B_{2,1}; \mZ)
 \stackrel{\partial_{\ast}}{\longrightarrow}
 H_1(\partial B_{22,15} \cup \partial B_{16,11} \cup \partial B_{4,1}
 \cup \partial B_{2,1}; \mZ) \stackrel{i_{\ast}}{\longrightarrow}
 H_1(B_{22,15} \cup B_{16,11} \cup B_{4,1} \cup B_{2,1};\mZ)$
 is not surjective. In particular, we claim that $(0,0,0,1)$ is not
 in the image of $i_{\ast} \circ \partial_{\ast}: H_2(Z_0, \partial B_{22,15}
 \cup \partial B_{16,11} \cup \partial B_{4,1} \cup \partial B_{2,1}; \mZ)
 \longrightarrow H_1(B_{22,15} \cup B_{16,11} \cup B_{4,1} \cup B_{2,1};\mZ)
 = \mZ/22\mZ \oplus \mZ/16\mZ \oplus \mZ/4\mZ \oplus \mZ/2\mZ$.\\

\n {\em Claim: $(0,0,0,1) \not \in Im(i_{\ast} \circ
 \partial_{\ast})$} - Suppose that an element $(0,0,0,1)$ is in
 $Im(i_{\ast} \circ \partial_{\ast})$.
 Then, by Lemma~\ref{lem-3} and Lemma~\ref{lem-4} above,
 an element $(0,0,0,1)$ should be a linear combination of
 $\{(10,0,0,0), (0,13,3,0), (0,0,2,0), (19,0,0,1), (1,0,3,0),
 (1,13,0,0),$ $(1,0,0,1), (19, 0, 1,0), (21,1,0,0), (0,12,0,0) \}$
 over integers.
 In other words, there should exist a set of integers
 $\{a_1, a_2, \ldots, a_{10}\}$ satisfying the following four linear equations:
\begin{eqnarray*}
 10a_1 + 19a_4 + a_5 + a_6 + a_7 + 19a_8 +21a_9 & \equiv & 0 \pmod{22} \\
 13a_2 + 13a_6 + a_9 + 12a_{10} & \equiv & 0 \pmod{16} \\
 3a_2 +2a_3 + 3a_5 + a_8 & \equiv & 0 \pmod{4} \\
 a_4 + a_7 & \equiv & 1 \pmod{2}
\end{eqnarray*}

 But one can easily check that there is no solution satisfying the
 four linear equations above simultaneously because of the parity -
 As mod 2, the above four linear equations become
 \begin{eqnarray*}
 a_4 + a_5 + a_6 + a_7 + a_8 + a_9 & \equiv & 0  \\
 a_2 + a_6 + a_9  & \equiv & 0  \\
 a_2 + a_5 + a_8 & \equiv & 0  \\
 a_4 + a_7 & \equiv & 1 .
\end{eqnarray*}
 By adding the second and the third equations, we get
 $a_5 + a_6 + a_8 + a_9  \equiv 0$.
 But it is impossible  by the first and the fourth equations!

\m

 Finally, since $i_{\ast} \circ \partial_{\ast}$ is not surjective,
 we conclude from the commutative diagram above that
 $H_1(Z_{22,16,4,2}; \mZ) \neq 0$. Hence Proposition~\ref{pro-1}
 implies that $H_1(Z_{22,16,4,2}; \mZ)$ $=\mZ/2\mZ$.
\end{proof}

\b

\section{A surface with $K^2=2$ and $H_1=\mZ/3\mZ $}
\label{sec-4}

 In this section we construct a surface of general type with
 $p_g=0, K^2=2$ and $H_1=\mZ/3\mZ$ using the same technique.
 As we mentioned in the Introduction,
 the key ingredient in the construction of such a surface
 is to find a right rational elliptic surface $Z'$.
 Once we have a right rational elliptic surface $Z'$ for $K^2=2$ and
 $H_1=\mZ/3\mZ $, the remaining argument is the same as before.
 Hence we describe here only how to get such a rational elliptic surface
 $Z'$ which is following.

\m

\n{\bf Main Construction.} We begin with a rational elliptic surface
 $Y'=\mP^2\sharp 9\overline{\mP}^2$ which has one $I_8$-singular
 fiber, one $I_2$-singular fiber, and two nodal singular fibers
 used in Section 3 in~\cite{PPS} (Figure~\ref{Fig-YY}).
\begin{figure}[hbtb]
 \begin{center}
 \setlength{\unitlength}{1mm}
 \includegraphics[height=3.5cm]{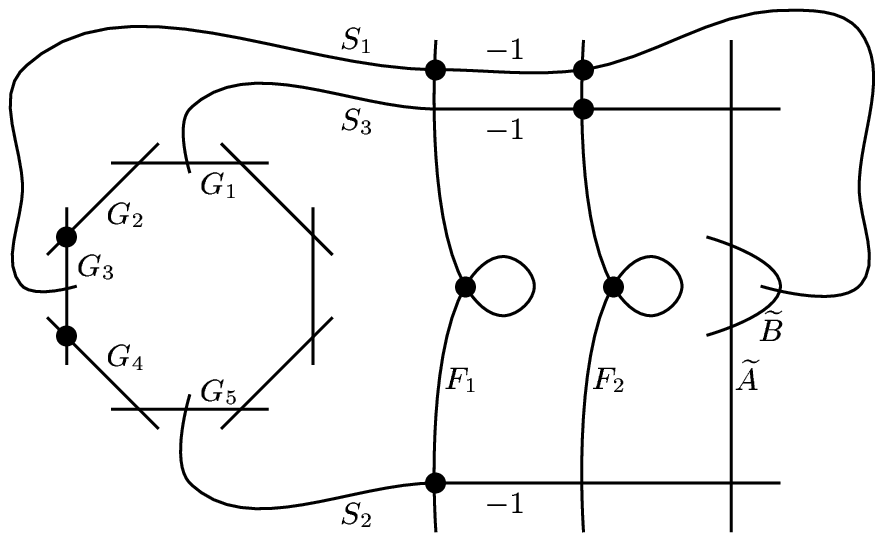}
 \end{center}
 \vspace*{-1 em}
 \caption{A rational elliptic surface $Y'=\mP^2\sharp 9\overline{\mP}^2$}
 \label{Fig-YY}
\end{figure}
 From this rational elliptic surface $Y'$, we first blow up at two singular
 points in the nodal fibers $F_1, F_2$ on $Y'$.
 Then the proper transforms $\t F_1, \t F_2$ of $F_1, F_2$ will be
 rational $(-4)$-curves whose homology classes are $[\t F_1]=[F_1]-2e_1$
 and $[\t F_2]=[F_2]-2e_2$, where $e_1, e_2$ are new exceptional curves
 in $Y'\sharp 2\overline{\mP}^2$ coming from two blowing ups.
 Next, we blow up at two black circled points lying in an $I_8$-singular
 fiber and we also blow up at four black circled points which
 are intersection points between sections $S_i$ and nodal fibers
 $F_i$ in Figure~\ref{Fig-YY} above.
 We blow up again twice at the intersection point between the proper
 transform of $F_i$ and the exceptional curve $e_i$ in the total transform
 of $F_i$ for $i=1,2$.
 Then we get a rational surface $Z':= Y' \sharp 10{\overline \mP}^2$
 which contains three disjoint configurations:
 $C_{9,5}={\overset{-2}{\circ}}-{\overset{-7}{\circ}}
 -{\overset{-2}{\circ}}-{\overset{-2}{\circ}}-{\overset{-3}{\circ}}$
 (which consists of $\t e_1, \t F_1, \t S_3, G_1, \t G_2$),
 $C_{9,5}={\overset{-2}{\circ}}-{\overset{-7}{\circ}}
 -{\overset{-2}{\circ}}-{\overset{-2}{\circ}}-{\overset{-3}{\circ}}$
 (which consists of $\t e_2, \t F_2, \t S_2, G_5, \t G_4$)
 and $D_{2,3}={\overset{-2}{\circ}}-{\overset{-3}{\circ}}
 -{\overset{-4}{\circ}}$ (which consists of
 $\t B, \t S_1, \t G_3$) (Figure~\ref{Fig-ZZ}).
 Note that the boundary of a configuration $C_{9,5}$ is a lens space
 $L(81,-44)$ which also bounds a rational ball $B_{9,5}$, and
 the boundary of a configuration $D_{2,3}$ is a lens space
 $L(18,-11)$ which bounds a Milnor fiber $M_{2,3}$.
 It is well known that the Milnor fiber $M_{2,3}$ is not a rational ball
 but a negative definite $4$-manifold with second Betti number $1$
 (cf.~\cite{Man}).

\begin{figure}[hbtb]
 \begin{center}
 \setlength{\unitlength}{1mm}
 \includegraphics[height=4 cm]{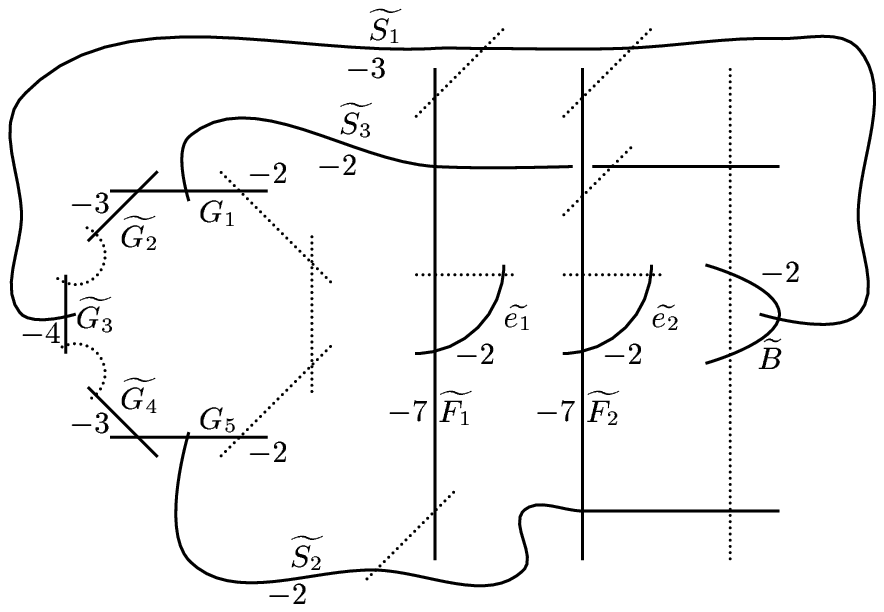}
 \end{center}
 \vspace{-1 em}
 \caption{A rational surface $Z'= Y' \sharp 10{\overline \mP}^2$}
 \label{Fig-ZZ}
\end{figure}

\m

 Then we contract these three disjoint chains $C_{9,5},\, C_{9,5},\, D_{2,3}$
 of $\mP^1$'s from $Z'$ so that it produces a normal projective
 surface, denoted by $X$, with three permissible singularities.
 Using the same technique as in~\cite{LP}, we are able to prove that
 $X$ has a $\mQ$-Gorenstein smoothing and a general fiber $X_t$ of
 the $\mQ$-Gorenstein smoothing is a minimal
 complex surface of general type with $p_g=0$ and $K^2=2$.
 Furthermore, the general fiber $X_t$ is diffeomorphic to a rational
 blow-down $4$-manifold $Z'_{9,9,2}$ which is obtained from
 $Z'= Y' \sharp 10{\overline \mP}^2$ by replacing three disjoint
 configurations $C_{9,5},\, C_{9,5}$ and $D_{2,3}$
 with corresponding Milnor fibers $B_{9,5},\, B_{9,5}$ and $M_{2,3}$
 respectively.
 Finally, using a similar technique in Section 3,
 it is easy to prove that the rational blow-down $4$-manifold
 $Z'_{9,9,2}$ has $H_1(Z'_{9,9,2}; \mZ)=\mZ/3\mZ$,
 which is Theorem~\ref{thm-main2}.

\b
\b


\begin{thebibliography}{999}

\bibitem[1]{BHPV} W. Barth, K. Hulek, C. Peters, A. Van de Ven, Compact
                  complex surfaces. 2nd ed. Springer-Verlag, Berlin, 2004.
\bibitem[2]{Cam} L. Campedelli, {\it Sopra alcuni piani doppi notevoli con curva
                  di diramazioni del decimo ordine}, Atti Acad. Naz. Lincei
                  \textbf{15} (1931), 536--542.
\bibitem[3]{LP}  Y. Lee and J. Park, {\it A simply connected surface
                 of general type with $p_g=0$ and $K^2=2$},
                 Invent. Math. \textbf{170} (2007), 483--505.
\bibitem[4]{Man} M. Manetti, {\it On the moduli space of diffeomorphic
                algebraic surfaces}, Invent. Math. {\textbf 143} (2001), 29--76.
\bibitem[5]{MP} M. Mendes Lopes and R. Pardini, {\it Numerical Campedelli
                surfaces with fundamental group of order 9}, to appear in J.
                European  Math. Soc. (JEMS).
\bibitem[6]{PPS} H. Park, J. Park and D. Shin, {\it A simply connected
                  surface of general type with $p_g=0$ and $K^2=3$},
                  arXiv:0708.0273.
\bibitem[7]{Reid} M. Reid, {\it Surfaces with $p_{g}=0, K_{S}^{2}=2$},
                 preprint available at
                 http: // www. maths. warwick. ac. uk /$\sim$masda/surf/.

\end{thebibliography}
\end{document}